\documentclass{article}
\usepackage[T2A]{fontenc}
\usepackage[utf8]{inputenc}
\usepackage[english]{babel}%

\usepackage{amsmath,amssymb,mathrsfs}
\usepackage{amsthm}
\usepackage{hyperref}
\usepackage{color}
\usepackage{cases}
\theoremstyle{plain}
\newtheorem{theorem}{Theorem}
\newtheorem{problem}{Problem}

\newtheorem{lemma}{Lemma}
\theoremstyle{definition}
\newtheorem{remark}{Remark}

\DeclareMathOperator{\curl}{curl}
\DeclareMathOperator{\const}{const}
\renewcommand{\Re}{\mathop{\mathrm{Re}}}
\renewcommand{\Im}{\mathop{\mathrm{Im}}}

\begin{document}

\begin{center}\Large\bfseries On the reconstruction of parameters of a moving fluid from the Dirichlet-to-Neumann map\end{center}
\begin{center}\it Alexey Agaltsov\footnote{CMAP, Ecole Polytechnique, France}\footnote{Moscow State University, Russia}\end{center}
\begin{center}\today\end{center}

\begin{quote}
We consider an inverse boundary value problem for a model time-harmonic equation of acoustic tomography of moving fluid with variable current velocity, sound speed, density and absorption. In the present article it is assumed that at fixed frequency the coefficients of this equation are already recovered modulo an appropriate gauge transformation using some reconstruction method from boundary measurements presented in the literature. Our main result consists in formulas and equations that allow to get rid of this gauge non-uniqueness and recover the fluid parameters using boundary measurements at several frequencies.
\medskip


\textbf{Subjects}: partial differential equations, mathematical physics

\end{quote}

\section{Introduction}
We consider a model equation for a time-harmonic acoustic pressure $\psi = \psi(x)$ (time dependence $e^{-i\omega t}$) in a moving fluid with sound speed $c = c(x)$, current velocity $v = v(x)$, density $\rho = \rho(x)$ and absorption $\alpha = \alpha(x,\omega) = \omega^{\zeta(x)}\alpha_0(x)$ at fixed frequency $\omega$:
\begin{gather}\label{in.waveeq}
  L_\omega \psi = 0 \quad \text{in $D$},\\
 \quad L_\omega = -\Delta - 2i A_\omega(x) \cdot \nabla - U_\omega(x), \quad x = (x_1,\ldots,x_d), \label{in.schrop}\\
  \nabla = \bigl( \tfrac{\partial}{\partial x_1},\ldots, \tfrac{\partial}{\partial x_d} \bigr), \quad \Delta = \tfrac{\partial^2}{\partial x_1^2} + \cdots+ \tfrac{\partial^2}{\partial x_d^2}, \notag
\end{gather} 
where
\begin{gather}
  \begin{gathered}\label{in.Dpro}
    \text{\it $D$ is an open bounded domain in $\mathbb R^d$,}\\
    \text{\it $d \in \{2,3\}$, with connected $C^\infty$ boundary $\partial D$}, 
  \end{gathered}\\
  A_\omega(x) = \frac{\omega v(x)}{c^2(x)} + i \nabla \ln \rho^{\frac 1 2}(x), \quad U_\omega(x) = \frac{\omega^2}{c^2(x)} + 2i\omega^{1+\zeta(x)}\frac{\alpha_0(x)}{c(x)}. \label{in.defAU}
\end{gather}
Note that the operator $L_\omega$ is a special case of the so-called magnetic Schroedinger operator.

The model equation \eqref{in.waveeq} was studied in different particular cases in \cite{Henk1988,Rous1994,Rych1996,Roum2009en,Bur2013en,Agal2015b,Agal2015c}.

In the present article we suppose that
\begin{subequations}
\begin{align}
  c & \in C^2(\overline D), \quad c > 0 \; \text{\it in $\overline D$}, \label{in.cpro}\\
  v & \in C^2(\overline D,\mathbb R^d), \label{in.vpro}\\
  \rho & \in C^{2,\beta}(\overline D), \; \beta\in(0,1], \quad \rho > 0 \; \text{\it in $\overline D$,} \label{in.rpro}\\
  \zeta & \in C(\overline D), \; \alpha_0 \in C(\overline D), \quad \text{\it $\zeta > 0$, $\alpha_0$ is real in $\overline D$}, \label{in.abpro}
\end{align}
\end{subequations}
where $\overline D = D \cup \partial D$ and $C^{k,\beta}(\overline D)$ denotes the space of $k$ times continuously differentiable functions in $\overline D$ whose $k$-th derivatives are $\beta$-H\"older continuous. Let
\begin{equation}
  \widetilde{c} = c|_{\partial D}, \quad \widetilde v = v|_{\partial D}, \quad \widetilde{\rho} = \rho|_{\partial D}.
\end{equation}
In what follows we always assume that 
\begin{equation}\label{in.unsol}
    \text{\it $0$ is not a Dirichlet eigenvalue for operator $L_\omega$ in $D$.}
\end{equation}
Note that the set of $\omega$'s for which \eqref{in.unsol} does not hold is locally finite. Besides, \eqref{in.unsol} always holds for $\omega = 0$.

For equation \eqref{in.waveeq} under the assumption \eqref{in.unsol} we consider the Dirichlet-to-Neumann type operator $\Lambda_\omega$ which maps a sufficiently regular function $f$ on $\partial D$ to the function
\begin{equation*}
  \Lambda_\omega f = \bigl(\tfrac{\partial \psi}{\partial \nu} + i(\nu \cdot A_\omega) \psi \bigr)\bigr|_{\partial D},
\end{equation*}
where $\psi$ is the solution of equation \eqref{in.waveeq} in $D$ with Dirichlet boundary condition $\psi|_{\partial D} = f$ and $\nu$ is the unit exterior normal field to $\partial D$.

Note that it is possible to get rid of the assumption \eqref{in.unsol} by considering a general Robin-to-Robin map instead of the Dirichlet-to-Neumann map, see \cite{Isay2013}.

We consider the following problem:
\begin{problem}\label{in.probD2N} Find $c$, $v$, $\rho$ and $\alpha$ in $D$ from $\Lambda_\omega$ given for $\omega$ in some fixed set $\Omega$ and from $\widetilde c$, $\widetilde v$ and $\widetilde \rho$.
\end{problem}
This problem was studied, in particular, in \cite{Agal2015b,Agal2015c}. In these works it was shown that
\begin{enumerate}
 \item[(a)] if $\rho \equiv \const$, $\alpha_0 \equiv 0$ and $\Omega = \{\omega_1\}$ then the Problem \ref{in.probD2N} is uniquely solvable;
 \item[(b)] if $\alpha_0 \equiv 0$, $\Omega = \{\omega_1,\omega_2\}$, $\omega_1 < \omega_2$, then the Problem \ref{in.probD2N} is uniquely solvable;
 \item[(c)] if $\Omega = \{\omega_1,\omega_2,\omega_3\}$, $\omega_1 < \omega_2 < \omega_3$, and $\zeta \neq 0$ in $D$ then the Problem \ref{in.probD2N} is uniquely solvable.
\end{enumerate}

Reconstruction results for Problem \ref{in.probD2N} at fixed $\omega$ can be summarized as follows:
\begin{enumerate}
 \item Uniqueness modulo gauge transformations in the most general case follows from the results of \cite{Guil2011} in dimension $d=2$ and of \cite{Krup2014} in dimension $d=3$.
 \item An approximate reconstruction algorithm modulo gauge transformations in dimension $d=2$ was developed in \cite{Agal2014b,Agal2015a}.
 \item A possible reconstruction approach modulo gauge transformations in dimension $d=3$ was outlined in \cite{Henk1988}.
\end{enumerate}

In the present article we assume that the coefficients $A_\omega$ and $U_\omega$ of equation \eqref{in.waveeq} are already recovered from $\Lambda_\omega$ up to a gauge transformation using some appropriate method at fixed $\omega$. The goal of the present work is to show how to get rid of the gauge non-uniqueness and recover $c$, $v$, $\rho$ and $\alpha$ using boundary measurements at several frequencies. In this respect, the present article can be considered as a development of the article \cite{Agal2015b} where the particular case $\rho \equiv \const$, $\alpha_0 \equiv 0$, $\Omega = \{\omega_1\}$ was considered, and of the article \cite{Agal2015c} where the corresponding uniqueness theorems for equation \eqref{in.waveeq} with general operator of the form \eqref{in.schrop} were obtained.

For a vector field $V = (V_1,\ldots,V_d)$ and a function $f$ in $D$ we set by definition
\begin{equation}
  \begin{aligned}
  \curl V & = \begin{cases}
              \partial_1 V_2 - \partial_2 V_1, \quad d = 2, \\
              \bigl( \partial_2 V_3 - \partial_3 V_2, \partial_3 V_1 - \partial_1 V_3, \partial_1 V_2 - \partial_2 V_1 \bigr), \quad d = 3, \\
            \end{cases}\\
  \curl f & = \bigl( \partial_2 f, -\partial_1 f\bigr), \quad d = 2,
  \end{aligned}
\end{equation}
where $\partial_j = \partial/\partial_{x_j}$.

In the present article it is assumed that the following functions are already recovered from the operator $\Lambda_\omega$ at fixed $\omega$:
\begin{align}
  F & = \curl \bigl( \tfrac{v}{c^2} \bigr) \quad \text{\it in $D$}, \label{in.defF}\\
  q_\omega & =  f_1 - \omega^2 f_2 + i \omega f_3 - 2i\omega^{1+\zeta}\tfrac{\alpha_0}{c} \quad \text{\it in $D$}, \label{in.defq}
\end{align}
where
\begin{equation}\label{in.deff}
  f_1 = \rho^\frac{1}{2} \Delta \rho^{-\frac 1 2}, \quad f_2 = \tfrac{1}{c^2}+\tfrac{v}{c^2}\cdot\tfrac{v}{c^2}, \quad  f_3 = \nabla \cdot \bigl(\tfrac{v}{c^2} \bigr) - \tfrac{v \cdot \nabla \ln \rho}{c^2}.
\end{equation}
For the corresponding identifiability results see \cite{Guil2011} (for $d = 2$) and \cite{Krup2014} (for $d = 3$); for an approximate reconstruction algorithm see \cite{Agal2014b,Agal2015a}.

Thereby, in the present article we study the following problems.
\begin{problem}\label{in.probgtna} Find $c$, $v$ and $\rho$ in $D$ from $q_\omega$ given for $\omega \in \Omega = \{\omega_1,\omega_2\}$, $\omega_1 < \omega_2$, and from $F$, $\widetilde c$, $\widetilde v$, $\widetilde \rho$. 
\end{problem}
\begin{problem}\label{in.probgtab} Find $c$, $v$, $\rho$, $\zeta$ and $\alpha_0$ in $D$ from $q_\omega$ given for $\omega \in \Omega = \{\omega_1,\omega_2,\omega_3\}$, $\omega_1 < \omega_2 < \omega_3$, and from $F$, $\widetilde c$, $\widetilde v$, $\widetilde \rho$. 
\end{problem}

\section{Solution of Problem \ref{in.probgtna}}
We are going to derive the explicit formulas for solving the Problem \ref{in.probgtna}. We consider \eqref{in.defq} with $\omega \in \Omega$ as a system of linear equations for $f_1$, $f_2$ and $f_3$. Solving this system we obtain
\begin{equation}\label{in.expf}
  f_1 = \frac{\omega_2^2 \Re q_{\omega_2} - \omega_1^2 \Re q_{\omega_1}}{\omega_2^2 - \omega_1^2}, \; f_2 = \frac{\Re q_{\omega_1} - \Re q_{\omega_2}}{\omega_2^2-\omega_1^2}, \; f_3 = \omega_1^{-1} \Im q_{\omega_1}.
\end{equation}
Set $g = \rho^{-\frac 1 2}$. It follows from formula \eqref{in.deff} that $g$ satisfies the following equation
\begin{gather}\label{in.rhoeq}
  g(x) = g_0(x) + \int_D G(x,y) f_1(y) g(y) \, dy, \quad x \in \overline D,\\
  g_0(x) = \int_{\partial D} \frac{\partial G(x,y)}{\partial \nu_y} \tilde \rho^{-\frac 1 2}(y) \, dy,
\end{gather}
where $G(x,y)$ is the (non-positive) Dirichlet Green's function for operator $\Delta$ in $D$ and $\nu_y$ is the unit exterior normal to $\partial D$ at point $y$. Note that $g_0$ is just the the harmonic extension of $\widetilde \rho^{-\frac 1 2}$ to $D$. The existence of function $G$ follows from assumption \eqref{in.Dpro} and from \cite[Theorem 4.17, p. 112]{Aub1982}. 

\begin{lemma} Equation \eqref{in.rhoeq} is uniquely solvable for $g \in C(\overline D)$. 
\end{lemma}
\begin{proof} Suppose that $g_1$, $g_2 \in C(\overline D)$ are two solutions of equation \eqref{in.rhoeq}. Then their difference $h = g_1 - g_2$ satisfies
\begin{equation}\label{in.heq}
  h(x) = \int_D G(x,y) f_1(y) h(y) \, dy, \quad x \in \overline D.
\end{equation}
Using formulas \eqref{in.rpro} and \eqref{in.deff} we obtain that $f_1 \in C^{0,\beta}(D)$. Taking this into account and using formula \eqref{in.heq} and Lemma \cite[Lemma 4.2]{Gilb2001}, we obtain that $h \in C^2(D) \cap C(\overline D)$ and that
\begin{subequations}
\begin{align}
  & -\Delta h + f_1(x) h = 0 \quad \text{\it in $D$}, \label{sol.hdeq} \\ 
  & h|_{\partial D} = 0.  \label{sol.hbnd}
\end{align}
\end{subequations}
Using formulas \eqref{in.expf} and \eqref{sol.hdeq} we can rewrite equation \eqref{sol.hdeq} as 
\begin{equation}\label{sol.hdeq2}
  -\Delta (\rho^{\frac 1 2} h) + \nabla \rho \cdot \nabla (\rho^{\frac 1 2} h) = 0, \quad \text{\it in $D$}.
\end{equation}
It follows from \cite[Lemma 4.2 and Corollary 8.2]{Gilb2001} and from formulas \eqref{sol.hbnd} and \eqref{sol.hdeq2} that $h \equiv 0$. Hence, $\rho^{-\frac 1 2}$ is the unique solution of class $C(\overline D)$ to equation~\eqref{in.rhoeq}. 
\end{proof}

For vectors $a = (a_1,\ldots,a_d)$ and $b = (b_1,\ldots,b_d)$ in $D$ we put by definition
\begin{equation}
  a \times b = \begin{cases}
                  a_1 b_2 - a_2 b_1, \quad d = 2,\\
                  \bigl( a_2 b_3 - a_3 b_2, a_3 b_1 - a_1 b_3, a_1 b_2 - a_2 b_1 \bigr), \quad d = 3.
               \end{cases}
\end{equation}
Recall that the Helmholtz decomposition of the vector field $\tfrac{v}{c^2}$ is given by the following formula:
\begin{equation}\label{sol.Helm}
  \frac{v}{c^2} = \nabla \Phi - \curl V \quad \text{\it in $D$},
\end{equation}
where 
\begin{gather}
  \Phi(x) = \int_D G_0(x-y) \nabla_y \cdot \bigl(\tfrac{v(y)}{c^2(y)}\bigr)\,dy -\int_{\partial D} G_0(x-y) \frac{\nu_y \cdot \widetilde v(y)}{\widetilde c^2(y)} \, dy,\\
  V(x) = \int_D G_0(x-y) F(y) \, dy - \int_{\partial D} G_0(x-y) \frac{\nu_y \times \widetilde v(y)}{\widetilde c^2(y)} \, dy, \label{in.defV}\\
  G_0(x) = -\frac{1}{2\pi}\ln|x|, \quad d = 2,\\
  G_0(x) = \frac{1}{4\pi} \frac{1}{|x|}, \quad d = 3,
\end{gather}
where $\nu_y$ is the unit exterior normal to $\partial D$ at point $y$. Note that the vector field $V$ is known since $F$, $\widetilde v$ and $\widetilde c$ are given. 

Using formula \eqref{sol.Helm} we can recover the function $\Phi|_{\partial D}$ modulo an additive constant (which does not matter). Fix a point $x^0 \in \partial D$. Let $x \colon [0,1] \to \partial D$ be a smooth curve linking $x^0$ to some given point $x \in \partial D$. Then
\begin{equation}\label{sol.PhiOnBnd}
  \Phi(x)-\Phi(x^0) = \int_0^1 \bigl(\tfrac{\widetilde v}{\widetilde c^2} + \curl V\bigr)|_{x(t)} \cdot \dot x(t) \, dt, \quad \dot x = \frac{dx}{dt}.
\end{equation}

It follows from formulas \eqref{in.deff} and \eqref{sol.Helm} that $\Phi$ satisfies the equation
\begin{equation*}
  \begin{gathered}
  -\Delta \Phi + \nabla \ln \rho \cdot \nabla \Phi = -f_3 + \curl V \cdot \nabla \ln \rho \quad \text{\it in $D$}, \\
  \text{\it or} \; -\Delta \eta + f_1(y) \eta =  \rho^{-\frac 1 2} (-f_3 + \curl V \cdot \nabla \ln \rho) \quad \text{\it in $D$},
  \end{gathered}
\end{equation*}
where $\eta = \rho^{-\frac 1 2}(\Phi-\Phi(x^0))$. The function $\eta$ can be found from the following integral equation:
\begin{gather}
  \eta(x) = \eta_0(x) + \eta_1(x) + \int_D G(x,y) f_1(y) \eta(y) \, dy, \quad x \in \overline D, \label{in.etaeq}\\
  \eta_0(x) = \int_D G(x,y) \rho^{-\frac 1 2}(y) \bigl( f_3(y)-\curl V(y) \cdot \nabla \ln \rho(y) \bigr) \, dy, \notag\\
  \eta_1(x) = \int_{\partial D} \frac{\partial G(x,y)}{\partial \nu_y} \tilde \rho^{-\frac 1 2}(\Phi(y)-\Phi(x^0)) \, dy, \notag
\end{gather}
where $G(x,y)$ is the (non-positive) Dirichlet Green's function for $\Delta$ in $D$. Note that $\eta_1$ is just the harmonic extension of $\tilde \rho^{-\frac 1 2}(\Phi|_{\partial D}-\Phi(x^0))$ to $D$. Also note that equation \eqref{in.etaeq} has the same kernel as equation \eqref{in.rhoeq}. Therefore, it is uniquely solvable for $\eta \in C(\overline D)$.

After recovering $\Phi - \Phi(x^0)$ in $D$ we can find $\tfrac{v}{c^2}$ using formula \eqref{sol.Helm}. Finally, using formulas \eqref{in.deff} and \eqref{sol.Helm} we obtain
\begin{equation}\label{in.expcv}
  \frac{1}{c^2} = f_2 - (\nabla \Phi - \curl V)^2, \quad v = c^2 (\nabla \Phi - \curl V).
\end{equation}
The described algorithm for solving Problem \ref{in.probgtna} is summarized in the following theorem.

\begin{theorem}\label{propna} Suppose that $D$ satisfies \eqref{in.Dpro}, $\Omega = \{\omega_1,\omega_2\}$, $\omega_1 \neq \omega_2$ and \eqref{in.unsol} holds for all $\omega \in \Omega$, and suppose that $c$, $v$, $\rho$ satisfy \eqref{in.cpro}--\eqref{in.rpro}. Then Problem \ref{in.probgtna} can be solved as follows:
\begin{enumerate}
 \item Define $f_1$, $f_2$, $f_3$ and $V$ using formulas \eqref{in.expf} and \eqref{in.defV}.
 \item Find $g$ as the unique solution of class $C(\overline D)$ to equation \eqref{in.rhoeq}. Set $\rho = g^{-2}$.
 \item Fix $x^0 \in \partial D$ and find $\Phi|_{\partial D}-\Phi(x^0)$ using formula \eqref{sol.PhiOnBnd}.
 \item Find $\eta$ as the unique solution of class $C(\overline D)$ to equation \eqref{in.etaeq}. Set $\Phi - \Phi(x^0) = \rho^{\frac 1 2}\eta$.
 \item Find $c$ and $v$ using the explicit formulas \eqref{in.expcv}.
\end{enumerate}
\end{theorem}

\begin{remark} Suppose that $d = 2$ and $D = \bigl\{ x \in \mathbb R^2 \mid |x| \leq 1\}$. Then the function $g_0$ from formula \eqref{in.rhoeq} and the function $G(x,y)$ from formulas \eqref{in.rhoeq} and \eqref{in.etaeq} can be found explicitly:
\begin{equation*}
  g_0(x) = \int\limits_{S^1} \widetilde \rho^{-\frac 1 2}(\vartheta) \frac{1-|x|^2}{|\vartheta - x|^2} \frac{d\vartheta}{2\pi}, \quad G(x,y) = \frac{1}{2\pi} \ln \frac{|x||y-x|}{\bigl|y |x|^2 - x\bigr|},
\end{equation*}
where $S^1 = \partial D$. Furthermore, if $\|f_1\|_{C(\overline D)} < 4$ then equations \eqref{in.rhoeq} and \eqref{in.etaeq} can be solved using the method of successive approximations in $C(\overline D)$.
\end{remark}

\section{Solution of Problem \ref{in.probgtab}}
Define the sets $D_0$ and $D_1$ by the formulas
\begin{equation}\label{in.defD0}
  D_0 = \bigl\{ x \in D \mid \tfrac{\Im q_{\omega_1}(x)}{\omega_1} = \tfrac{\Im q_{\omega_2}(x)}{\omega_2} \bigr\}, \quad D_1 = D \setminus D_0.
\end{equation}
It follows from formulas \eqref{in.abpro} and \eqref{in.defq} that $D_0 = \{ x \in D \mid \alpha_0(x) = 0 \}$ and that the functions $f_1$, $f_2$ and $f_3$ defined in \eqref{in.defq} can be found in $D_0$ using formulas \eqref{in.expf}.

Using formula \eqref{in.defq} for $x \in D_1$ and $\omega \in \Omega$ we obtain that
\begin{equation}\label{in.zetaeq}
  \frac{\omega_2^{-1}\Im q_{\omega_2}(x)-\omega_1^{-1}\Im q_{\omega_1}(x)}{\omega_3^{-1}\Im q_{\omega_3}(x) - \omega_1^{-1}\Im q_{\omega_1}(x)} = \frac{\bigl(\tfrac{\omega_2}{\omega_1}\bigr)^{\zeta(x)}-1}{\bigl(\tfrac{\omega_3}{\omega_1}\bigr)^{\zeta(x)}-1}.
\end{equation}
We consider \eqref{in.zetaeq} as an equation for finding $\zeta(x)$ at fixed $x \in D_1$. This equation is uniquely solvable for $\zeta(x)$ at fixed $x \in D_1$ as the following lemma shows.
\begin{lemma} The right side of equation \eqref{in.zetaeq} at fixed $x \in D_1$ is a strictly decreasing function of $\zeta(x) \in (0,+\infty)$. 
\end{lemma}
\begin{proof} It is sufficient to show that for any $p$, $q$ such that $1 < p < q$ and for any $t > 1$ the equality
\begin{equation}\label{in.ptqt}
  \frac{p^t-1}{q^t-1} = \frac{p-1}{q-1}.
\end{equation}
can not hold. Assuming the equality \eqref{in.ptqt}, we define
\begin{equation}\label{in.deflam}
  \lambda_1 = \frac{p-1}{q-1}, \quad \lambda_2 = \frac{q-p}{q-1}.
\end{equation}
Using formulas \eqref{in.ptqt} and \eqref{in.deflam} we obtain the formulas
\begin{subequations}
\begin{gather}
    \lambda_1 + \lambda_2 = 1, \quad \lambda_1 > 0, \; \lambda_2 > 0, \label{in.lamsum}\\
    \lambda_1 q + \lambda_2 = p, \quad \lambda_1 q^t + \lambda_2 = p^t. \label{in.lamcc}
\end{gather}
\end{subequations}
Since the function $f(s) = s^t$ is strictly convex, relations \eqref{in.lamsum}--\eqref{in.lamcc} can not hold. Therefore the initial assumption that \eqref{in.ptqt} holds must be false. 
\end{proof}

Next we find functions $f_1$, $f_2$, $f_3$ and $\alpha_0/c$ in the domain $D_1$. It follows from formula \eqref{in.defq} that $f_1$ and $f_2$ in $D_1$ can be found using formulas \eqref{in.expf}. It also follows from \eqref{in.defq} that $f_3$ and $\alpha_0/c$ in $D_1$ can be found from the following formulas:
\begin{equation}\label{in.alexp}
  f_3 = \frac{\frac{\omega_1^\zeta}{\omega_2}\Im q_{\omega_2} - \frac{\omega_2^\zeta}{\omega_1} \Im q_{\omega_1}}{\omega_1^\zeta - \omega_2^\zeta}, \quad \frac{\alpha_0}{c} = \frac 1 2\frac{\omega_2^{-1}\Im q_{\omega_2} - \omega_1^{-1} \Im q_{\omega_1}}{\omega_1^\zeta - \omega_2^\zeta}.
\end{equation}
Using the values of $f_1$, $f_2$ and $f_3$ in $D$ we can find $c$, $\rho$ and $v$ in $D$ using the steps 2--5 mentioned in Theorem \ref{propna}. Finally, we find $\alpha_0$ using $\alpha_0/c$ and $c$. The algorithm for solving Problem \ref{in.probgtab} is summarized in the following proposition.

\begin{theorem}\label{propab} Suppose that $D$ satisfies \eqref{in.Dpro}, $\Omega = \{\omega_1,\omega_2,\omega_3\}$, $\omega_1 < \omega_2 < \omega_3$, and \eqref{in.unsol} holds for all $\omega \in \Omega$. Suppose also that $c$, $v$, $\rho$, $\zeta$ and $\alpha_0$ satisfy \eqref{in.cpro}--\eqref{in.abpro}. Then Problem \ref{in.probgtab} can be solved as follows:
\begin{enumerate}
 \item Define $D_0$ and $D_1$ using formula \eqref{in.defD0}. Find $f_1$ and $f_2$ in $D$ using formulas \eqref{in.expf}. Find $f_3$ using formula \eqref{in.expf} in $D_0$ and formula \eqref{in.alexp} in $D_1$.
 \item Find $\zeta(x)$ at fixed $x \in D_1$ as the unique positive solution to equation \eqref{in.zetaeq}.
 \item Find $\alpha_0/c$ in $D_1$ using formula \eqref{in.alexp}.
 \item Find $g$ as the unique solution of class $C(\overline D)$ to equation \eqref{in.rhoeq}. Set $\rho = g^{-2}$.
 \item Fix $x^0 \in \partial D$ and find $\Phi|_{\partial D}-\Phi(x^0)$ using formula \eqref{sol.PhiOnBnd}.
 \item Find $\eta$ as the unique solution of class $C(\overline D)$ to equation \eqref{in.etaeq}. Set $\Phi-\Phi(x^0) = \rho^{\frac 1 2}\eta$.
 \item Find $c$ and $v$ using the explicit formulas \eqref{in.expcv}. Set $\alpha_0$ to zero in $D_0$ and find $\alpha_0$ from $\alpha_0/c$ and $c$ in $D_1$.
\end{enumerate}
\end{theorem}

\begin{remark} Note that the formulas and equations presented in Theorem \ref{propna} (resp. \ref{propab}) require the knowledge of the function $q_\omega$ at two (resp. three) frequencies $\omega$. These formulas are exact but they can be not very stable with respect to the noise in the initial data. However, if $q_\omega$ is known for a bigger number of frequencies it is possible to increase the stability of reconstruction by replacing formulas \eqref{in.expf} and \eqref{in.alexp} with their least squares analogues. A numerical study of reconstruction stability will be carried out in a subsequent paper.
\end{remark}

\section{Aknowledgements}
The present article was prepared in the framework of research conducted under the supervision of Prof. R. G. Novikov.

\bibliographystyle{alpha}
\bibliography{biblio_utf}

\begin{thebibliography}{RBKS09}

\bibitem[Aga15a]{Agal2015a}
A.~D. Agaltsov.
\newblock {Finding scattering data for a time-harmonic wave equation with first
  order perturbation from the Dirichlet-to-Neumann map}.
\newblock {\em Journal of Inverse and Ill-Posed Problems}, 23(6):627--645,
  2015.

\bibitem[Aga15b]{Agal2015b}
A.~D. Agaltsov.
\newblock A global uniqueness result for acoustic tomography of moving fluid.
\newblock {\em Bulletin des Sciences Mathematiques}, 139(8):937--942, 2015.

\bibitem[AN14]{Agal2014b}
A.~D. Agaltsov and R.~G. Novikov.
\newblock {R}iemann-{H}ilbert problem approach for two-dimensional flow inverse
  scattering.
\newblock {\em J. Math. Phys.}, 55(10), 2014.
\newblock id 103502.

\bibitem[AN15]{Agal2015c}
A.~D. Agaltsov and R.~G. Novikov.
\newblock Uniqueness and non-uniqueness in acoustic tomography of moving fluid.
\newblock {\em Journal of Inverse and Ill-Posed Problems}, 2015.
\newblock doi:10.1515/jiip-2015-005.

\bibitem[Aub82]{Aub1982}
T.~Aubin.
\newblock {\em {Nonlinear analysis on manifolds. Monge-Amp\`ere equations}}.
\newblock Springer-Verlag, New York, 1982.

\bibitem[BSZR13]{Bur2013en}
V.~A. Burov, A.~S. Shurup, D.~I. Zotov, and O.~D. Rumyantseva.
\newblock Simulation of a functional solution to the acoustic tomography
  problem for data from quasi-point transducers.
\newblock {\em Acoustical Physics}, 59(3):345--360, 2013.

\bibitem[GT01]{Gilb2001}
D.~Gilbarg and N.~S. Trudinger.
\newblock {\em Elliptic Partial Differential Equations of Second Order}.
\newblock Classics in Mathematics. Springer-Verlag, Berlin Heidelberg, 2001.

\bibitem[GT11]{Guil2011}
C.~Guillarmou and L.~Tzou.
\newblock {Identification of a connection from Cauchy data on a Riemann surface
  with boundary}.
\newblock {\em Geom. Funct. Anal.}, 21(2):393--418, 2011.

\bibitem[HN88]{Henk1988}
G.~M. Henkin and R.~G. Novikov.
\newblock A multidimensional inverse problem in quantum and acoustic
  scattering.
\newblock {\em Inv. Problems}, 4:103--121, 1988.

\bibitem[IN13]{Isay2013}
M.~I. Isayev and R.~G. Novikov.
\newblock {Reconstruction of a potential from the impedance boundary map}.
\newblock {\em Eurasian Journal of Mathematical and Computer Applications},
  1(1):5--28, 2013.

\bibitem[KU14]{Krup2014}
K.~Krupchyk and G.~Uhlmann.
\newblock Uniqueness in an inverse boundary problem for a magnetic
  {S}chr{\"o}dinger operator with a bounded magnetic potential.
\newblock {\em Comm. Math. Phys.}, 327(3):993--1009, 2014.

\bibitem[RBKS09]{Roum2009en}
O.~D. Rumyantseva, V.~A. Burov, A.~L. Konyushkin, and N.~A. Sharapov.
\newblock Increased resolution of two-dimensional tomography imaging along the
  transverse coordinate and separate reconstruction of elastic and viscous
  scatterer characteristics.
\newblock {\em Acoustical Physics}, 55(4):613--629, 2009.

\bibitem[RE96]{Rych1996}
M.~N. Rychagov and H.~Ermert.
\newblock Reconstruction of fluid motion in acoustic diffraction tomography.
\newblock {\em J. Acoust. Soc. Am.}, 99(5):3029--3035, 1996.

\bibitem[RW94]{Rous1994}
D.~Roussef and K.~B. Winters.
\newblock Two-dimensional vector flow inversion by diffraction tomography.
\newblock {\em Inv. Problems}, 10:687--697, 1994.

\end{thebibliography}

\end{document}